\documentclass[12pt,oneside,final]{amsart}
\usepackage[cp1250]{inputenc}
\usepackage[T1]{fontenc}
\usepackage{amssymb}
\usepackage{color}
\usepackage{url}
\renewcommand{\thefootnote}{}
\newtheorem{twr}{Twierdzenie}
\newtheorem{thr}[twr]{Theorem}
\newtheorem{prp}[twr]{Proposition}

\newtheorem{lm}[twr]{Lemma}

\newtheorem{crl}[twr]{Corollary}
\renewcommand{\thefootnote}{}
\newcommand{\da}{\partial}
\newcommand{\db}{\bar\partial}
\newcommand{\dc}{i{\partial}\bar\partial}
\newcommand{\ta}{\theta}
\newcommand{\tb}{\bar\theta}
\newcommand{\tc}{\theta\bar{\theta}}

\newcommand{\psh}{ \mathcal{PSH}}

\newcommand{\ma}{{(i\partial\bar\partial u)^2}}
\DeclareMathOperator\R{{ \mathbb R}}

\DeclareMathOperator\N{{ \mathbb N}}
\DeclareMathOperator\Q{{ \mathbb Q}}

\newcommand{\cn}{\mathbb{C}^n}

\title[capacity]{On the relative capacity on almost complex surface  }
\author[S. Pli\'{s}]{Szymon Pli\'{s}}
\address{  Institute of Mathematics, Cracow University of Technology, Warszawska 24, 31-155
    Krak\'{o}w, Poland
}
\address{  Institute of Mathematics Polish Academy of Sciences ul. \'Sniadeckich 8, 00-956 Warszawa, Poland }

\email{splis@pk.edu.pl}

\begin{document}\thispagestyle{empty} \footnotetext{The  author was partially supported by the NCN grant 2013/08/A/ST1/00312.}
\renewcommand{\thefootnote}{\arabic{footnote}} 

\begin{abstract}

We built the pluripotential theory on almost complex surfaces. Using Bedford-Taylor type relative capacities we prove among others that J-holomorphic curves as well as negligible sets are pluripolar and Josefson's type theorem on almost Stein surfaces.  
\end{abstract}

\maketitle
\section{Introduction}
Let $(M,J)$ be an almost complex surface (an almost complex manifold of  real dimension 4). For  functions $u,v\in W^{1,2}_{loc}(\Omega)$, where $\Omega\subset M$ is a domain, we can define, see \cite{p2}, a wedge product
\begin{equation*}\label{wedgeproduct}
\dc u\wedge\dc v\end{equation*}
as a $(2,2)$ current. 
If $u,v$ are $\mathcal{C}^2$ functions then it is the standard wedge product of continuous forms. If $u,v\in W^{1,2}_{loc}(\Omega)$ are plurisubharmonic then it is a regular Borel measure, see \cite{p2,p3} and $(\dc u)^2$ is called the Monge-Amp\`ere operator.

The goal of this article is to study plurisubharmonic functions, the Monge-Amp\`ere operator and the relative capacity on almost complex surfaces in the spirit of Bedford-Taylor \cite{b-t2} which (together with \cite{b-t1}) created pluripotential theory in $\cn$. 

In his famous paper \cite{g} M. Gromov shows a deep connection between analysis on almost complex manifolds and symplectic geometry. For the "generic" almost complex structure $J$ there is no non-constant $J$-holomorphic functions. On the other hand there is always  plenty of $J$-holomorphic discs which constitute a powerful tool for symplectic geometry. Plurisubharmonic functions are another objects
which locally always exist    and can be viewed as dual to $J$-holomorphic discs.  They  are very important  for symplectic applications too but it seems that they are much less understood. In the present paper we develop  the theory of this class of functions in the case of surfaces. 

 The main difference between $\cn$ and an almost complex manifold (with the non-integrable almost complex structure) is the fact that for a plurisubharmonic function $u$, the positive current $\dc u$ is not necessary closed. Thus the pluripotential theory  on almost complex manifold is in some sense similar to pluripotential theory on hermitian manifold where the current $\dc u+\omega$ is not closed too (for this subject see \cite{k2} and references there).
However, the theory in the non-integrable case is much more difficult. This is, among others, because in the case of hermitian manifolds the non-closed part of $\dc u+\omega$ is just the hermitian form $\omega$ which is smooth and does not depend on $u$ whereas in our situation  the non-closed part of $\dc u$   is only in $L_{loc}^2$ (at least for bounded $u$) and  strongly depends on $u$.

In \cite{r} J.-P. Rosay proved that $J$-holomorphic (smooth) curves are pluripolar (for more general result of any submanifold see \cite{e}, for application of the pluripolarity to theory of $J$-holomorphic discs see \cite{i-r}). We prove that any (not necessary smooth) $J$-holomorphic curve on an almost complex surface is pluripolar (Corollary \ref{curves}). The problem of the pluripolarity of singular submanifolds in higher dimensional manifolds is still open.

The most important result in the paper is the quasi-continuity of plurisubharmonic functions (Proposition \ref{Quasi-continuity} and Theorem \ref{Quasi-continuity2}) which allow us in particular to prove that  negligible sets are pluripolar (Theorem \ref{negligible})  and obtain further results about the pluripolarity (Corollary \ref{plpbycap2} and Theorem \ref{globalplp2}).
 
All results proved in the paper,  in the case of $\cn$, are proved in the classical paper \cite{b-t2}. However the author have learned the subject from \cite{b1} and \cite{K}. In particular, when it is possible, we follow (as closely as possible) Ko\l odziej's lecture notes.

\section{Preliminaries}
\subsection{Almost complex manifolds and plurisubharmonic functions}

We say that $(M,J)$ is an almost complex manifold if $M$ is a manifold and $J$ is an $\mathcal{C}^\infty$ smooth endomorphism of the tangent bundle $TM$, such that $J^2=-\rm{id}$. The real dimension of $M$ is even in that case. We will  denote by $n$ the complex dimension of $M$: $n=\dim_\mathbb{C}M=\frac{1}{2}\dim_\mathbb{R}M$. Starting from subsection \ref{Monge-Amp\`ere operator} we will assume that $n=2$ that is $M$ is a surface.

All definitions below are exactly the same as in the case of complex manifolds. 

As on complex manifolds we can define here $(p,q)$-forms and more generally $(p,q)$-currents. We have the decomposition of the exterior differential:
\begin{equation*}
    d=\da+\db-\ta-\tb,
\end{equation*}
where operators  $\da,\db,-\ta$ and $-\tb$ are respectively $(1,0)$, $(0,1)$, $(2,-1)$ and $(-1,2)$ parts of $d$. On the level of functions we have 
\begin{equation*}
    d=\da+\db.
\end{equation*}

   We always assume that we have fixed an almost hermiatian form $\Omega$ (a positive $(1,1)$-form) which gives us Rimmanian metric on $M$,  Sobolev   norms for functions and allows us to define $cap_\omega$. However it does not matter for us which $\omega$  we choose hear. In dimension $2$, at least locally, we can choose $\omega$ to be closed (which is not always  possible in higher dimension, for both facts see \cite{l}) but we will not use this fact.

Through the whole paper  $\Omega$ will be a relatively compact domain in $M$ and $\mathbb{D}=\{z\in\mathbb{C}\colon |z|<1\}$. In $\mathbb{C}$ we have the standard almost complex structure $J_{st}$. We say that a function $\lambda\colon\mathbb{D}\to D$ is $J$-holomorphic if $d\lambda J_{st}=Jd\lambda$. We say that a function $u\colon\Omega\to[-\infty,+\infty)$ is plurisubharmonic iff 
\begin{enumerate}
 \item $u\not\equiv -\infty$
    \item  $u$ is upper-semicontinuous and
    \item $u\circ\lambda$ is subharmonic function or $u\circ\lambda\equiv -\infty$ for any $J$-holomorphic function $\lambda\colon\mathbb{D}\to \Omega$.
\end{enumerate}

If $u$ is plurisubharmonic then it is locally integrable and $\dc u\geq 0$, see \cite{p}. The converse was proved by R. Harvey and B. Lawson in \cite{h-l}. Namely they proved that, if $u\in L^1_{loc}$ and $\dc u\geq 0$ then a function $\tilde u$, given by
\begin{equation*}
   \tilde u(z)= {\rm ess} \,\limsup_{w\to z\;\;\;\;}u(w)
\end{equation*}
is a plurisubharmonic function which is equal almost everywhere to the function $u$.

We say that a function $u$ on $\Omega$ is strictly plurisubharmonic iff for any  $\varphi\in\mathcal{C}^\infty_0(\Omega)$ there is $\varepsilon_0>0$ such that the function $u+\varepsilon\varphi$ is plurisubharmonic for $\varepsilon_0>\varepsilon>0$.

We say that a domain $\Omega\Subset M$ is strictly pseudoconvex (of class $\mathcal{C}^{\infty}$),  if there is a strictly plurisubharmonic function $\rho$ of class $\mathcal{C}^{\infty}$  in a neighborhood of $\bar\Omega$,  such that $\Omega=\{\rho<0\}$ and $\triangledown\rho\neq0$ on $\partial\Omega$. 
We say that $M$ is almost Stein if there is exhausting smooth strictly plurisubharmonic function on $M$. Of course any point on an almost complex manifold has a strictly pseudoconvex neighbourhood and  strictly pseudoconvex domains are almost Stein.

In \cite{b-t2} capacities are studied in hyperconvex domains. Here for simplicity we will assume that $\Omega $ is strictly pseudoconvex. However all results holds,  with the same proofs, in accordingly defined hyperconvex domains on almost complex manifolds. 

R. Harvey, B. Lawson and the author proved in \cite{h-l-p} the following result about regularization of plurisubarmonic functions.

\begin{thr}\label{regularisation}
Suppose  $M$ is an almost Stein
 and let $u $ be a plurisubharmonic function on $M$.
Then there exists a decreasing sequence $u_j $ of smooth strictly 
plurisubharmonic functions on $M$  such that $u_j\downarrow u(x)$.
\end{thr}

In the case of surfaces the existence of the local regularization was proved earlier in \cite{p3}.

\subsection{Dirichlet problem for the Monge-Amp\`ere equation }
Let $\Omega\subset M$ be strictly pseudoconvex domain. The followinng Theorem will be useful for us.
\begin{thr}\label{MAtheorem}
There is a unique  solution $u$ of the Dirichlet problem:
 \begin{equation}\label{DP}
\left\{
\begin{array}{l}
    u\in\mathcal{PSH}(\Omega)\cap \mathcal{C}^\infty(\bar\Omega)\\ 
    (i\partial\bar\partial u)^n=dV \;\mbox{ in }\;\Omega\\
    u=\varphi\;\mbox{ on }\;\partial\Omega
\end{array}
\right.\;,
\end{equation}
for any $\varphi\in\mathcal{C}^\infty(\bar\Omega)$ and $dV$ is the volume form on a neighbourhood of $\bar\Omega$. 
\end{thr}

The proof of this theorem in \cite{p1} has a gap (in the part about the interior second order estimate). It is corrected in a recent work of J. Chu, V. Tosatti and B. Weinkove \cite{c-t-w}.\footnote{However if we do the calculations from \cite{p1} in corrected way, we can relatively easily get Theorem \ref{MAtheorem} for small enough domains $\Omega$ and it is enough for applications in \cite{p2} and in the current article.}

For the existence of the weak solutions (in the viscosity sense) of (\ref{DP}) in the case of the continuous data  see \cite{h-l}.

\subsection{Monge-Amp\`ere operator}\label{Monge-Amp\`ere operator} From here, we will assume that $M$ is an almost complex surface.
 For $u,v\in W^{1,2}_{loc}(\Omega)$, where $\Omega\subset M$ is a domain, we can define a wedge product
\begin{equation*}\label{wedgeproduct}
\dc u\wedge\dc v\end{equation*}\begin{equation*}
:=-\dc(i\da u\wedge\db v)+\da(\da u\wedge\tb\da v)+\db(\ta\db u\wedge\db v)+\tc\da u\wedge\db v-\ta\db u\wedge\tb\da v\end{equation*}\begin{equation*}
=-\dc(i\da u\wedge\db v)+\da(\da u\wedge\tb\da v)+\db(\ta\db u\wedge\db v)-\tb\da u\wedge\ta\db v-\ta\db u\wedge\tb\da v,
\end{equation*}
as a $(2,2)$ current. 
As was mentioned before, if $u,v$ are $\mathcal{C}^2$ functions then it is the standard wedge product of continuous forms. If $u,v\in W^{1,2}_{loc}(\Omega)$ are plurisubharmonic then it is a regular Borel measure
and $(\dc u)^2$ is  the Monge-Amp\`ere operator of the function $u$. By the regularization we can see that the above wedge product is symmetric.

We can in a standard way define for $u\geq0$ (and so for all bounded $u$)
$$i\da u\wedge\db u\wedge\dc v=\frac{1}{2}\dc u^2\wedge\dc v-u\dc u\wedge\dc v.$$
By the smooth approximation and the Lebesgue monotone convergence theorem we get
$$i\da u\wedge\db u\wedge\dc v\geq0.$$

Note that a current $dd^cu$ is well defined  for $u$  plurisubharmonic and $$d^c=-dJ=i(\db-\da),$$
where the last equality holds on the level of functions. For smooth functions we can calculate that
$$(dd^cu)^2=(\dc u)^2+2\tb\da u\wedge\ta\db u$$
and it is clear that for bounded (or in $W^{1,2}_{loc}$) plurisubharmonic function $u$ it is a positive measure too. Moreover  it is an exact current
$$(dd^cu)^2=d(d^cu\wedge dd^cu)$$
so we can apply Stoke's theorem and even  prove a comparison principle for it (compare Proposition 29). We will exploit this facts, but we will not use  the notion $dd^c$ besides this subsection.  The main reasons that we prefer here to develop $(\dc u)^2$ is that, at least for smooth plurisubharmonic functions , $(\dc u)^2=0$ if and only if $u$ is maximal (see \cite{h-l} or \cite{p1}) and it is no longer true for the operator $(dd^cu)^2$. A similar property is that if $u_0$ is a fixed smooth strictly plurisubharmonic function in neighbourhood of $\bar\Omega$ then a smooth function $u$ is plurisubharmonic if and only if for any $t>0$ we have $(\dc (u+tu_0))^2>0$ and again it is not true for $(dd^cu)^2$.

To see that the case of the non-integrable almost complex structure is very different form complex manifold recall that on a holomorphic surface
$$\int_\Omega\varphi(dd^cu)^2=\int_\Omega udd^c\varphi\wedge dd^cu$$
for smooth $\varphi$ with compact support which does not holds in our situation.

\section{Basic estimates}

The following proposition was proved in \cite{p3}.
\begin{prp}[proposition 4.2 in \cite{p3}]\label{W12estimates} Let $u\in\mathcal{PSH}\cap{W}^{1,2}_{loc}(\Omega)$ then:\\
i) If $v\in\mathcal{PSH}(\Omega)$ and $v\geq u$,  then $v\in{W}^{1,2}_{loc}(\Omega)$;\\
ii) If a sequence $u_j$ of plurisubharmonic functions decreases to $u$, then it converges in ${W}^{1,2}_{loc}$.
\end{prp}

Note here that $i)$ imply that bounded plurisubharmonic functions are in $W^{1,2}$. From the proof in \cite{p1} of the above we also get
\begin{prp}\label{blockiw12} Let $D\Subset \Omega$, $u,v\in\mathcal{PSH}(\Omega)$, $u\leq v\leq0$ and $u\in{W}^{1,2}(\Omega)$. Then
 $v\in{W}^{1,2}(D)$ and $\|v\|_{W^{1,2}(D)}\leq C\|u\|_{W^{1,2}(\Omega)}$, where the constant $C$ depends only on $D$ and $\Omega$.
\end{prp}
Note that B\l ocki in \cite{b2} proved the above estimate for subharmonic functions in  $\mathbb{R}^n$. 

As a  direct consequence we get the following
\begin{crl}\label{oszacowaniedlaograniczonych}
 If $K\Subset\Omega$ and $u$ be a bounded plurisubharmonic function then
\begin{equation*}\label{boundedpsh}
\|u\|_{W^{1,2}(K)}\leq C\|u\|_{\Omega}.
\end{equation*}
\end{crl}

\begin{thr}[Chern-Levine-Nirenberg inequalities] 
Let  $K\Subset\Omega$ and $u,v\in\mathcal{PSH}\cap W^{1,2}(\Omega)$, then
\begin{equation*}\label{111}
\int_K\dc u\wedge\dc v\leq C\|u\|_{W^{1,2}(\Omega)}\|v\|_{W^{1,2}(\Omega)},
\end{equation*}
and if in addition $u,v$ are bounded then
\begin{equation*}\label{11}
\int_K\dc u\wedge\dc v\leq C\|u\|_{\Omega}\|v\|_{\Omega}.
\end{equation*}
\end{thr}

{\it Proof:} Take a nonnegative test function $\varphi$ which is equal $1$ on $K$. By definition of the current $\dc u\wedge\dc v$ and the integration by parts we can estimate:
$$\int_K\varphi\dc u\wedge\dc v$$
$$\leq\int_\Omega\left(-\dc\varphi\wedge i\da u\wedge\db v-d\varphi\wedge(\da u\wedge\tb\da v+\ta\db u\wedge\db v)\right)$$$$+\int_\Omega\varphi(\tc\da u\wedge\db v-\ta\db u\wedge\tb\da v)$$
$$\leq C\|u\|_{W^{1,2}(\Omega)}\|v\|_{W^{1,2}(\Omega)},$$
where $C$ depends on $\varphi$ and $J$. The second part follows  from the first one and Corollary \ref{oszacowaniedlaograniczonych}. $\;\Box$

\section{Convergence Theorem for increasing sequences}
As in the integrable case we define the relative capacity of the Borel subset $E$ of $\Omega$ as
$$cap(E)=cap(E,\Omega)=\sup\{\int_E(\dc u)^2:u\in\mathcal{PSH}(\Omega),-1\leq u\leq0\}.$$
We shall also consider the following set function associated to the form  $\omega$:
$$cap_\omega(E)=cap_\omega(E,\Omega)=\sup\{\int_E \dc u \wedge\omega
:u\in\mathcal{PSH}(\Omega),-1\leq u\leq0\}.$$
When $E\Subset\Omega$ then by the Chern-Levine-Nirenberg inequalities we have \begin{equation*}
    cap(E,\Omega)<+\infty
\end{equation*} and thus if there is a bounded function $h\in\mathcal{PSH}(\Omega)$ which satisfies $\omega\leq \dc h$ we also have $cap_\omega(E,\Omega)<+\infty$. We say that a sequence  $u_k$ of plurisubharmonic functions defined on $\Omega$  converges with respect to capacity to a function $u$ if for any compact set $K\subset\Omega$ and $t>0$
$$\lim_{k\rightarrow\infty}cap(K\cap\{|u-u_k|>t\},\Omega)=0.$$ In the same way we define the convergence with respect to $cap_\omega$.

\begin{prp}\label{capomega}Let $\Omega$ be a strictly pseudoconvex domain and let $u_k$ be a sequence of plurisubharmonic functions
which decreases on $\Omega$ to a bounded plurisubharmonic function u. Then it
converges with respect to $cap_\omega$.
\end{prp}

\begin{proof}
We can asume that all $u_k$ are equal outside a compact set $E\subset\Omega$. We fix $v\in\mathcal{PSH}(\Omega)$, $-1\leq v\leq 0$. Using integration by parts and Corollary \ref{oszacowaniedlaograniczonych} we can estimate
$$0\leq I_k=\int_\Omega (u_k-u)\dc v\wedge\omega$$
$$=-\int_\Omega i\da (u_k-u)\wedge\db v\wedge\omega+\int_\Omega i (u_k-u)\db v\wedge\da\omega$$
$$\leq C\|u_k-u\|_{W^{1,1}(K)}\|v\|_{W^{1,1}(K)}\leq C'\|u_k-u\|_{W^{1,1}(K)},$$
where constants $C$, $C'$ depend only on $\Omega$. Since our estimate does not depend on the function  $v$, by Propositions \ref{W12estimates}  we get that $I_k\to 0$ as $k\to\infty$ independently on $v$ and (as in \cite{K}) the Proposition follows.  
 \end{proof}

\begin{prp}Let $u$ be a bounded plurisubharmonic function on $\Omega$ and $\varepsilon>0$. Then, there exists an open set $U\subset\Omega$ with $cap_\omega(U,\Omega)<\varepsilon$ and such that $u$ restricted to $\Omega\setminus U$ is continuous.
\end{prp}

Using the regularization result (Theorem \ref{regularisation}) we can prove it exactly like in the case of domains in $\cn$ (see for example the proof of theorem 1.13 in \cite{K}).

Again exactly as in $\cn$, from the above Proposition we get

\begin{crl}\label{frazypradygdyzbiegaja}
Let $\mathcal{U}$ be a uniformly bounded family of plurisubharmonic functions in  $\Omega$. Suppose that $u,v\in\mathcal{U}$ and $(v_k)\subset\mathcal{U}$ and $$\dc v_k\to \dc v.$$ Then $$u\dc v_k\to u\dc v.$$
\end{crl}

\begin{lm}\label{w12cur}Let $\mathcal{U}$ be a uniformly bounded family of plurisubharmonic functions in  $\Omega$. Suppose that $(u_k),(v_k)\subset\mathcal{U}$  increase almost everywhere to $u$ and $v$ respectively. Then $$i\da u_k\wedge \db v_k\to i\da u\wedge \db v.$$
\end{lm}

\begin{proof}
First we will prove that
\begin{equation}\label{zb01}
u_k\db v_k\to u\db v.
\end{equation}
Let $\varphi\in\mathcal{D}_{(2,1)}$. Using Stokes theorem we can calculate
$$\int_\Omega u_k\db v_k\varphi-\int_\Omega u\db v\varphi$$
$$=\int_\Omega (u_k-u)\db v_k\varphi+\int_\Omega u\db( v_k-v)\varphi$$
$$=\int_\Omega (u_k-u)\db v_k\varphi+\int_\Omega ( v_k-v)\db(u\varphi).$$
Since $L^2$ norms of $\db v_k\varphi$ and $\db(u\varphi)$ depend only on $\varphi$ and $\mathcal{U}$, using the Cauchy-Schwarz inequality, we can choose a constant $C$ not depending on $k$ such that
$$\int_\Omega u_k\db v_k\varphi-\int_\Omega u\db v\varphi\leq C(\|u_k-u\|_{L^2(\Omega)}+\|v_k-v\|_{L^2(\Omega)})\to 0.$$
Thus (\ref{zb01}) follows.

The second step is to obtain the following convergence
\begin{equation}\label{zbfrazyidd}
u_k\dc v_k\to u\dc v.
\end{equation}
Let  $\varphi\in\mathcal{D}_{(1,1)}$ be positive. By Corollary \ref{frazypradygdyzbiegaja} we get 
\begin{equation}\label{oszlimsup}
\limsup_{k\to \infty}\int_\Omega u_k\dc v_k\wedge\varphi\leq\limsup_{k\to \infty} \int_\Omega u\dc v_k\wedge\varphi= \int_\Omega u\dc v\wedge\varphi.\end{equation}
Set $s\in\mathbb{N}$. Using Stokes' theorem 
we can estimate
$$\liminf_{k\to \infty}\int_\Omega u_k\dc v_k\wedge\varphi\geq\liminf_{k\to \infty} \int_\Omega u_s\dc v_k\wedge\varphi= \int_\Omega u_s\dc v\wedge\varphi$$
$$=\int v\dc u_s\wedge\varphi+\int u_si\db v\wedge\da \varphi+\int vi\da u_s\wedge\db\varphi.$$
From (\ref{zb01}) and again Corollary \ref{frazypradygdyzbiegaja} the last line with $s\to\infty$ converge to
$$\int v\dc u\wedge\varphi+\int ui\db v\wedge\da \varphi+\int vi\da u\wedge\db\varphi= \int_\Omega u\dc v\wedge\varphi.$$
This together with (\ref{oszlimsup}) gives us (\ref{zbfrazyidd}).

In the last step we will finish the proof. By (\ref{zb01}) and (\ref{zbfrazyidd}) we can conclude
$$i\da u_k\wedge \db v_k=i\da\left(u_k\db v_k\right)-u_k\dc v_k\to i\da\left(u\db v\right)-u\dc v= i\da u\wedge \db v.$$
\end{proof}

 The previous Lemma gives us easily the following two results.

\begin{crl}\label{w12dlarosnacych}
Suppose that $u_k$ is a sequence of locally bounded plurisubharmonic functions which increase to plurisubharmonic function $u$ almost everywhere.  Then $u_k$ converges to $u$ in $W^{1,2}_{loc}(\Omega)$.  
\end{crl}

\begin{thr}[Convergence Theorem for increasing sequences]\label{MAdlarosnacych}
Suppose that $u_k$ and $v_k$ are  sequences of locally bounded plurisubharmonic functions which increase to plurisubharmonic functions $u$ and $v$ respectively almost everywhere. Then 
$$\dc u_k\wedge\dc v_k\to \dc u\wedge\dc v.$$
\end{thr}

\section{Pluripolarity}\label{Pluripolarity}
In this section we prove the pluripolarity of $J$-holomorphic curves and some other facts about pluripolarity. More general results will be proved in section \ref{Quasi-continuity and applications} but we would like to show here that some important results can be proved without quasi-continuity what is a harder task.

\begin{prp}\label{0cap} Assume that $E\subset\Omega\subset M$. Then
$$E\, plp\Rightarrow cap (E,\Omega)=0.$$
\end{prp}

\begin{proof}
We can assume that  there is a compactly supported non-negative function $\varphi$ which is equal $1$ on $E$. Let $U\in\mathcal{PSH}$ be such that  $U|_E=-\infty$. Set $v\in\mathcal{PSH}$ such that $-1\leq v\leq 0$. Since $U\in L^1_{loc}$  the sequence $U/k$ increases to $0$ almost everywhere on the open set $\Omega'=\{U<0\}$. Thus the sequence $v_k=\max \{U/k,v\}$ increases to $0$ almost everywhere too. From the convergence Theorem we obtain $(\dc v_k)^2\to 0$. On the other hand, on the open set $\Omega_k=\{U<k\}$ we have $v_k=v$. Thus 
$$\int_E(\dc v)^2\leq\int_{\Omega'}\varphi(\dc v_k)^2\to 0,$$
and we can conclude that $\int_E(\dc v)^2=0$. Because this holds for all $v$ from the definition of the capacity we get that $cap(E)=0$.
\end{proof}

There is a constant $c_0$ (which depends only on $\Omega$) such that \begin{equation}\label{c_0}
\ta\db\varphi\wedge\tb\da\varphi\leq c_0i\da\varphi\wedge\db\varphi\wedge\omega
\end{equation} for any smooth function $\varphi$ defined on $\Omega$.
  To prove the next result about the pluripolarity we need the following Lemma:

\begin{lm}\label{prawiemaksymalnosc}
Let $h\in\mathcal{C}(\bar\Omega)$ be  such that \begin{equation*}
\dc h\geq 9c_0\omega \mbox{ on } \Omega \mbox{ and } \liminf_{z\to\da\Omega} h(z)\geq0 .
\end{equation*}
Let $D=\{h<0\}$. For $u\in \mathcal{PSH}\cap L^{\infty}(D)$ satisfying
\begin{equation*}
\inf u=-1,\;\liminf_{z\to\da D} u(z)\geq0 \mbox{ and } (\dc u)^2=0 \mbox{ on } D
\end{equation*}
we have $u\geq h$.
\end{lm}

\begin{proof}
Set $\varepsilon>0$. Since every connected component of $D$ is  an almost Stein manifold, by Theorem \ref{regularisation} there exists  a  sequence $u_k$ of smooth plurisubharmonic functions on $D$, which decreases to $u+\varepsilon$. Assume that there is $k_1\in\N$ such that the set $\{u_{k_1}<\rho\}$ is not empty. By convergence theorem for decreasing sequences there exist $k_2\in\N$ such that for $k\geq k_2$ we have \begin{equation}\label{nierownosczaproksymacji}
\int_{\{u<-\varepsilon\}}(\dc u_{k})^2<\frac{1}{9}\int_{\{u_{k_1}<h\}}(\dc h)^2\leq\frac{1}{9}\int_{\{u_{k}<h\}}(\dc h)^2.
\end{equation}
Set $k\geq k_2$ and put $v=u_{k}$, $\tilde v=(v+1)^2-1$, $E=\{v<0\}$, $\tilde E=E\cap\{\frac{h+\tilde v}{3}>v\}$, $F=\{v<h\}$. Since $2v\geq \tilde v$ on $E$ and $\tilde v=v$ on $\da E$ we have
\begin{equation*}
F\subset\tilde E\Subset E.
\end{equation*}
For small enough $\delta>0$ we still have $$E_\delta=E\cap\{\frac{h+\tilde v}{3}>v-\delta\}\Subset E.$$

Let us choose $\delta$ such that the set $\da E_\delta$ has Lebesgue measure equal $0$ and put $\varphi=v-\delta$ and $\psi=\max\{\varphi,\frac{h+\tilde v}{3}\}$.

Using (\ref{nierownosczaproksymacji}), the assumption about $\dc h$ and the inequality $$\dc \tilde v\geq2i\da v\wedge\db v$$
we can estimate
\begin{equation*}
\int_E(dd^c\varphi)^2<\frac{1}{9}\int_E(\dc h)^2+2\int_{E_\delta}\ta\db v\wedge\tb\da  v+\int_{E\setminus \bar E_\delta}(dd^c\psi)^2
\end{equation*}
\begin{equation*}
\leq\frac{1}{9}\int_E(\dc h)^2+2c_0\int_{E_\delta}i\da v\wedge\db  v\wedge\omega+\int_{E\setminus \bar E_\delta}(dd^c\psi)^2
\end{equation*}
\begin{equation*}
\leq \frac{1}{9}\int_{E_\delta}(\dc h)^2+\frac{2}{9}\int_{E_\delta}i\da v\wedge\db  v\wedge\dc h+\int_{E\setminus \bar E_\delta}(dd^c\psi)^2\leq \int_E(dd^c\psi)^2.
\end{equation*}
But this inequality contradicts with Stokes theorem. This gives us that $F$ is empty for any choose of  $k\in\N$ and $\varepsilon>0$. We thus get $u\geq h$.
\end{proof}

The following Lemma, which gives us the uniqueness of  some very special Dirichlet problem is a key step to prove the pluripolarity of curves as well as further  results related to the pluripolarity.
\begin{lm}\label{CP}
Let $\Omega$ be a strictly pseudoconvex domain. Let $u\in L^\infty\cap\mathcal{PSH}(\Omega)$ be such that $\lim_{z\to\partial\Omega}u(z)=0$ and $(\dc u)^2=0$. Then $u=0$ in $\Omega$.
\end{lm}

\begin{proof}
To prove the Lemma by contradiction let us assume that $u\neq0$. Put $u_1=\frac{u}{\|u\|_{L^\infty(\Omega)}}$ and $u_{k+1}=2u_k+1$ for $k\geq 1$. We can choose the  defining function $h_1$ for $\Omega$ such that $\dc h_1\geq9c_0\omega$. Let $h_{k+1}=k_k+\frac{1}{2}=h_1+\frac{k}{2}$ and $D_k=\{z\in\Omega:h_k<0\}$. By Lemma \ref{prawiemaksymalnosc} and the induction we easily get that $h_k\leq u_k$. On the other hand $\inf u_k=-1$ and $\inf h_k\to\infty$. Contradiction! 
\end{proof}

For an open set $V\subset M$ and a subset $E\subset V$ we put
$$u_E=u_{E,V}=\sup\{v\in\mathcal{PSH}(V):v\leq 0\mbox{ and } v|_E\leq-1\}.$$
The function $u_E$ is the relative extremal function well known in complex analysis in $\cn$. It was studied  on almost complex manifolds by Sukhov in \cite{s}\footnote{More precisly Sukhov considers the
plurisuperharmonic measure  $\omega_\star(\cdot,E,\Omega)=-u_E^\star$.
}.

\begin{lm}\label{support}
A function $u_E^\star$ is plurisubharmonic and ${\rm supp}\; (\dc u_E^\star)^2\subset \partial E$.
\end{lm}

\begin{proof}

By the Choquet lemma there is an increasing sequence of plurisubharmonic functions $u_j\geq-1$ with $(\lim u_j)^\star=u_E^\star$. Since $$\dc(\lim u_j)\geq0$$ the function $(\lim u_j)^\star$ is plurisubharmonic  and the Lebesgue measure of the set $\{\lim u_j\neq(\lim u_j)^\star\}$ is equal to $0$.

Let $p\in V\setminus E$. There is a domain $D\subset V\setminus E$ which is a smooth strictly pseudoconvex neighborhood of $p$. For $j\in\N$ let $\varphi^{(j)}_k$ be a sequence of smooth functions which decreases to $u_j$ on $\partial D$. By Theorem \ref{MAtheorem}, we can solve Dirichlet problem:
$$\begin{cases}
w^{(j)}_k\in\mathcal{C}^\infty(\bar D)\cap\mathcal{PSH}(D),\\
(\dc w^{(j)}_k)^2=k^{-1}\omega^2,\\
w^{(j)}_k|_{\partial D}=\varphi^{(j)}_k.
\end{cases}$$
Put $$w_j=\begin{cases}
u_j & \mbox{ on } V\setminus D,\\
\lim_{k\to\infty}w^{(j)}_k & \mbox{ on }  D.
\end{cases}$$
Then $w_j$ is a sequence of plurisubharmonic functions increasing almost everywhere  to $u_E$. Moreover by the convergent theorem for decreasing sequences $(\dc w_j)^2=0$ on $D$ and thus by the convergent theorem for increasing sequences $(\dc u_E^\star)^2=0$ on $D$. But we can choose $D$ as a neighborhood of any point in $V\setminus E$ what gives us that ${\rm supp}\; (\dc u_E^\star)^2\subset \partial E$. 
\end{proof}

\begin{prp}\label{plp}
Let $\Omega$ be a strictly pseudoconvex domain. Assume that $E$ is $F_\sigma$ subset of $\Omega$ and $cap (E)=0$. Then $E$ is pluripolar. Moreover there is a negative plurisubharmonic function $u$ on $\Omega$ such that $u|_E=-\infty.$
\end{prp}

\begin{proof}

Let $E_i$ be an increasing sequence of compact subsets such that $\bigcup E_i=E$. Put $w_i=u^\star_{E_i}$. By Lemma \ref{support} we get that $(\dc w_i)^2=0$. Because $\Omega$ is strictly pseudoconvex we have $\lim_{z\to\partial\Omega}w_i(z)=0$ and by Lemma \ref{CP} we obtain $w_i=0$.

Similar as in Lemma \ref{support}, by the Choquet lemma, for any $i$ there is an increasing sequence of plurisubharmonic functions $v_k^{(i)}$ such that $\lim_{k\to\infty}v_k^{(i)}=0$ almost everywhere and $v_k^{(i)}\leq-1$ on $E_i$. By the Lebesgue's Monotone Convergence Theorem we can choose for any $i$ a number $k$ such that $\|h_i\|_{L^1(\Omega)}\leq\frac{1}{2^i}$ for $h_i=v_k^{(i)}$.
We can conclude that a function $$u=\sum_{i=1}^\infty h_i$$
is negative, 
plurisubharmonic and $u|_E=-\infty$.
\end{proof}

\begin{crl}\label{curves}
For any $J$-holomorphic function $u:\mathbb{D}\to M$, a set $u(\mathbb{D})$ is pluripolar. 
\end{crl}

\begin{proof}
 Let $p\in u(\mathbb{D})$. We can choose a strictly pseudoconvex neighbourhood $U$ of $p$. Let $E=u(\mathbb{D})\cap U$. The function $u$ has at most countable many singular points (see for example lemma 2.7 in \cite{m}). Thus using Rosay theorem about pluripolarity of smooth $J$-holomorphic curves,  we get that $E$ is a sum of countable many compact pluripolar sets. This implies that $cap(E,U)=0$ and by Proposition \ref{plp} $E$ is pluripolar. Thus the Corollary follows.
\end{proof}

\begin{prp}\label{globalplp} 
Let $M$ be an almost stein manifold and let $E\subset M$ be pluripolar $F_\sigma$ set. Then there is a plurisubharmonic function $u$ on $\Omega$ such that $u|_E=-\infty$.
\end{prp}

\begin{proof}
Let $\rho$ be an exhaustion smooth strictly plurisubharmonic function on $M$. By the Sard's theorem there is a sequence $(a_k)\subset\R$ for which $a_{k+1}\geq a_k+1$ and all connected components of $$\Omega_k=\{z\in M:\rho(z)<a_k\}$$ are strictly pseudoconvex. Like in the proof of Proposition \ref{plp} we can choose a  function $u_k\in\mathcal{PSH}(\Omega_k)$ such that $-1\leq u_k\leq0$, $u_k|_{E\cap\Omega_k=-1}$  and $\|u_k\|_{L^1(\Omega_k)}<\frac{1}{2^k}$. Put $$v_k=\begin{cases}\max\{\rho-a_{k+1},u_{k+2}\} & \hbox{ on } \Omega_{k+2}, \\
\rho-a_{k+1} & \hbox{ on } M\setminus \Omega_{k+2},
\end{cases}$$
and $u=\sum v_k$. Since $v_k=u_{k+2}$ on $\Omega_k$ it is clear that $u$ has required properties.
\end{proof}

Comparing this results to Rosay theorem from \cite{r} we obtain the pluripolarity  of larger class of functions and by Proposition \ref{globalplp} we can choose plurisubharmonic function in a possible larger domain. On the other hand from the point of view of applications (see \cite{i-r}) it seems that it is more important to have the complete pluripolarity of sets which is not obtained by our methods.

\section{Convergence in capacity for decreasing sequences}

This section is the most technical part of the paper. 

We will need the following lemma.
\begin{lm}\label{niekonieczniepsh} Let $u,v\in W^{1,2}(\Omega)$ be such that  $u=0$ outside some compact subset of $\Omega$ and $\dc u\wedge\dc v$ is of order $0$, then  \begin{equation*}\label{1111} \left|\int_\Omega\dc u\wedge\dc v\right|\leq C\|u\|_{W^{1,2}(\Omega)}\|v\|_{W^{1,2}(\Omega)}. \end{equation*}\end{lm}

{\it Proof:} By Stokes' theorem we get
$$\int_\Omega\dc u\wedge\dc v=\int_\Omega\tc\da u\wedge\db v-\ta\db u\wedge\tb\da v,$$ and thus the statement follows. $\;\Box$

For a locally bounded plurisubharmonic function $w$ we will use the notion $\langle u,v\rangle_w$ for the real part of the expression
\begin{equation*}\label{uvw}
    i\da u\wedge\db v\wedge\dc w.
\end{equation*}
It is well defined, in particular when $u,v\in\psh\cap L^\infty_{loc}$ and in that case
$$\langle u,u\rangle_w\geq 0$$
and 
$$\langle u,v\rangle_w=\frac{1}{2}\left(\langle u+v,u+v\rangle_w-\langle u,u\rangle_w-\langle v,v\rangle_w\right).$$
We can linearly extend $\langle \cdot,\cdot\rangle_w$ for any linear combination of plurisubharmonic functions but at the moment the positivity of $\langle u-v,u-v\rangle_w$ is not clear.\footnote{As well as the independence of $\langle x,y\rangle_w$ from the representation of $x$ and $y$ as linear combinations of plurisubharmonic functions. But we will consider only the case where $x$ and $y$ are linear combinations of two fixed plurisubharmonic functions and it is clear how to calculate $\langle \cdot,\cdot\rangle_w$.} 

\begin{prp}\label{utimesv}
Let $u,v,w\in \psh\cap L^\infty_{loc}$. Let $(u_k)$, $(v_k)\subset\psh$ be sequences of plurisubharmonic functions which decrease to $u$, $v$ respectively. Then
$$\langle u_k,v_k\rangle_w\to \langle u,v\rangle_w.$$
\end{prp}

\begin{proof}
Let us first assume that $v$ is smooth. In the same way as for 
$\langle u,u\rangle_w$ we get that $$\langle u-v,u-v\rangle_w\geq 0.$$
This gives us the Cauchy-Schwarz inequality
 $$\int_E\langle x,y\rangle_w\leq \sqrt{\int_E\langle x,x\rangle_w}\sqrt{\int_E\langle y,y\rangle_w},$$
for $x,y$ which are linear combinations of $u$ and $v$ and $E $ a compact set.

Let $u_k$ be a  sequence of smooth plurisubharmonic functions which decreases to $u$. Then 
$$0\leq \langle u_k-u,u_k-u\rangle_w$$
$$=\frac{1}{2}\dc(u_k-u)^2\wedge\dc w-(u_k-u)\dc (u_k-u)\wedge\dc w$$
$$\leq\frac{1}{2}\dc(u_k-u)^2\wedge\dc w+(u_k-u)\dc u\wedge\dc w.$$
By Proposition \ref{1111} the first term converges to 0 and by the monotone convergence theorem the second term converge to 0 too and thus $$\langle u_k-u,u_k-u\rangle_w\to 0.$$ Using the Cauchy-Schwarz inequality we get that 
\begin{equation}\label{convergenceofproducts}
\langle u_k,u_k\rangle_w\to \langle u,u\rangle_w.\end{equation}
Thus by the  regularization  we get that  for $x$, $y$  any linear combinations of plurisubharmonic functions $\langle x,y\rangle_w$ is well defined (it not depend on the way how we represent $x$ and $y$ as a linear combinations) and
$$\langle x,x\rangle_w\geq 0.$$

Using this  we can now repeat the proof of (\ref{convergenceofproducts}) for not necessary smooth $u_k$ and by definition of  $\langle \cdot,\cdot\rangle_w$ the Proposition follows.
\end{proof}

\begin{lm}\label{L2mix}Let $u$, $v\in\psh\cap L^\infty_{loc}(\Omega)$. Then
 $\da u\wedge\da v\in L^2_{loc}(\Omega)$ and  $L^2$ norm of it on compact subset $E\subset\Omega$ depends only on $E$, $\|u\|_{L^\infty(\Omega)}$ and $\|v\|_{L^\infty(\Omega)}$. Moreover, if $(u_k)\subset\psh(\Omega)$ decreases to $u$ then  $\da u_k\wedge\da v$ converges to $\da u\wedge\da v$ in $L^2_{loc}(\Omega)$.
\end{lm}

\begin{proof}
Since the statement is local we can assume that $\Omega$ is strictly pseudoconvex and $1\geq u,v\geq 0$. Let $E$ be any compact subset of $\Omega$. To prove that $\da u\wedge\da v\in L^2_{loc}(\Omega)$ it is enough to show that
$$\int_E i\da u\wedge\db u\wedge i\da v\wedge\db v<+\infty.$$
This can be easily estimate in the following way.
$$ 0\leq i\da u\wedge\db u\wedge i\da v\wedge\db v= \left( \frac{1}{2}\dc u^2-u\dc u\right)\wedge \left( \frac{1}{2}\dc v^2-v\dc v\right)$$
$$ \leq \frac{1}{4}\dc u^2\wedge\dc v^2$$
and we get 
$$0\leq\int_E i\da u\wedge\db u\wedge i\da v\wedge\db v\leq cap(E,\Omega).$$

To prove the convergence of currents we will estimate from above the following quantity 
$$I=I(k)=\int_E i\da (u_k-u)\wedge\db(u_k-u)\wedge i\da v\wedge\db v\geq 0 ,$$ where $E\subset\Omega$ is compact.
Similarly as above we can estimate
$$ i\da (u_k-u)\wedge\db(u_k-u)\wedge i\da v\wedge\db v$$ $$\leq\left( \frac{1}{2}\dc (u_k-u)^2-(u_k-u)\dc (u_k-u)\right)\wedge\dc v^2$$
$$\leq\frac{1}{2}\dc (u_k-u)^2\wedge\dc v^2+(u_k-u)\dc u\wedge\dc v^2$$
and thus
$$I\leq\frac{1}{2}\int_E\dc (u_k-u)^2\wedge\dc v^2+ \int_E(u_k-u)\dc u\wedge\dc v^2.$$
The first therm converges to 0 by Proposition \ref{W12estimates}{\it ii)} and Lemma \ref{niekonieczniepsh} and the second one by the Lebesgue monotone convergence theorem.
\end{proof}

For $v,w\in\psh\cap L^\infty(\Omega)$ let us put
$$S=S(v,w)=\db(\da v\wedge\db w)+\da v\wedge\tb\da w+\ta\db v\wedge\db w.$$
The current $S$ is of order 1 \footnote{Actually $S$ is of order 0 but we will not use it.} and we have
$$dS=-\dc(i\da v\wedge\db w)+\da(\da v\wedge\tb\da w)+\db(\ta\db v\wedge\db w)$$ $$= \dc v\wedge\dc w+\tb\da v\wedge\ta\db w+\ta\db v\wedge\tb\da w.$$
For $u\in\psh\cap L^\infty(\Omega)$, by Lemma \ref{L2mix}, the current $uS$ can be well defined as current of order 1 too, as 
$$uS(\varphi)=\int-\da v\wedge\db w\wedge\db(u\varphi)+ u\da v\wedge\tb\da w\varphi+ u\ta\db v\wedge\db w\varphi.$$
We also put 
$$du\wedge S=\dc v\wedge i\da u\wedge\db w+ \da u\wedge\da v\wedge\tb\da w+ \db u\wedge\ta\db v\wedge\db w.$$
Again by Lemma \ref{L2mix} it is well defined current of order 0.

We need the next Lemma to be able to use Stokes theorem for the current $uS$.

\begin{lm}\label{doStokesa}Let  $u,v,w\in\psh\cap L^\infty(\Omega)$ and let $S=S(v,w)$ be as above. Then
$$d(uS)=du\wedge S+udS.$$ In particular the current $d(uS)$ is of order 0.
\end{lm}

\begin{proof}
If $u $ is smooth the Lemma is a consequence of formula for $dS$. For arbitrary $u$ we can (locally) choose a decreasing sequence $(u_k)\subset\psh\cap \mathcal{C}^\infty$ which converges to $u$ and by Proposition  \ref{utimesv} and Lemma \ref{L2mix} we obtain the Lemma.\end{proof}

\begin{lm}\label{kolejnylematozbieznosci}Let  $u\in\psh\cap L^\infty_{loc}(\Omega)$ and let $(u_k)\subset\psh(\Omega)$ be a sequence which decreases to $u$. Then  for any compact subset $E$, there is a sequence of positive numbers $\eta_k$ converging to $0$,  such that 
$$\int_Ei\da (u_k-u)\wedge\db (u_k-u)\wedge \dc v\leq\eta_k,$$
for all $v\in\psh$, $0\leq v\leq 1$.\end{lm}

\begin{proof}
Since the statement is local we can assume that
$\Omega$
is strictly
pseudoconvex and
outside  some compact subset of $\Omega$ all functions $u_k$ are equal $u$. The quantities $I_1,I_2,\ldots$ below depend on $k$.

Let $$I_1=\int_\Omega\dc u\wedge i\da (u_k-u)\wedge\db(u_k-u)\geq 0$$
    and 
    $$I_2=\int_\Omega (u_k-u)\dc u\wedge \dc u_k\geq 0$$
    We have $$I_1=\frac{1}{2}\int_\Omega\dc u\wedge \dc (u_k-u)^2+\int_\Omega(u_k-u)\dc u \wedge\dc(u-u_k)$$
    and thus $$I_1+I_2= \frac{1}{2}\int_\Omega\dc u\wedge \dc (u_k-u)^2+\int_\Omega(u_k-u)\ma.$$
Using Lemma \ref{niekonieczniepsh} we get that the first term in the last line is bounded by $C\|u_k-u\|_{W^{1,2}(\Omega)}$  and the second one converges to $0$ independently on $v$ too. Thus  $I_1,I_2$ converge to $0$ as $k\to0$.

In the similar way we can show convergence of the following sequence of integrals

$$I_3=\int_\Omega\dc u_k \wedge i\da (u_k-u)\wedge\db(u_k-u)\geq 0.$$
Let us calculate
 $$I_3=\frac{1}{2}\int_\Omega\dc u_k\wedge \dc (u_k-u)^2+\int_\Omega(u_k-u)\dc u_k \wedge\dc(u-u_k)$$
    $$=\frac{1}{2}\int_\Omega\dc u_k\wedge \dc (u_k-u)^2+\int_\Omega(u_k-u)\dc u_k \wedge\dc u$$
    and again the first term is bounded by $C\|u_k-u\|_{W^{1,2}(\Omega)}$  and the second is equal to $I_2$.
    
Let $w\in\{u,u_k\}$. Put
$$I_5=\int_\Omega{\rm Re}\dc w\wedge i\da (u_k-u)\wedge\db  v.$$
By the Cauchy-Schwarz inequality we
 can estimate
$$|I_5|\leq \sqrt{\int_\Omega\dc w\wedge i\da  v\wedge\db  v}\sqrt{\int_\Omega\dc w\wedge i\da (u_k-u)\wedge\db (u_k-u)}$$
    $$\leq C\sqrt{\int_\Omega\dc w\wedge i\da (u_k-u)\wedge\db (u_k-u)}.$$
    Its gives us that for
    $$I_6=\int_\Omega{\rm Re}\dc(u_k-u)\wedge i\da (u_k-u)\wedge\db  v$$
    we have
    $$|I_6|<C(\sqrt{I_1}+\sqrt{I_2}), $$ which again gives us the convergence of this quantity to $0$.
    
    Put $$I_7=\int_\Omega{\rm Re}\left(\db(u_k-u)\wedge\ta\db(u_k-u)\wedge\db  v\right.$$ $$+(u_k-u)\left.(\tb\da (u_k-u)\wedge\ta\db  v+\ta\db (u_k-u)\wedge\tb\da  v)\right).$$ Using Lemma \ref{L2mix} and the Cauchy-Schwarz inequality we get that
    $$I_7\leq C\|u_k-u\|_{W^{1,2}(\Omega)}.$$
    
 Put $$I_8=\int_\Omega(u_k-u)\dc(u-u_k)  \wedge\dc  v.$$
Using Stokes theorem we get
$$I_8=-I_6-I_7\to 0.$$

Finally 
$$ \int_\Omega i\da (u_k-u)\wedge\db (u_k-u)\wedge \dc  v$$
$$=\frac{1}{2}\int_\Omega \dc (u_k-u)^2\wedge \dc  v+I_8.$$
Again by Lemma \ref{niekonieczniepsh} the first term can be estimated by $C\|u_k-u\|_{W^{1,2}(\Omega)}$ and thus we get 
the convergence to $0$.
\end{proof}

\begin{lm}\label{L2mixzb}Let  $u,v\in\psh(\Omega)$ and $0\leq u,v\leq 1$. Let $(u_k)\subset\psh(\Omega)$ be a sequence which decreases to $u$. Then  a sequence $\da u_k\wedge\da v$ converges to $\da u_k\wedge\da v$ in $L^2_{loc}$ independently on $v$. More precisely, for any compact subset $E$, there is a sequence of positive numbers $\eta_k$ decreasing to $0$, which does not depend on $v$, such that
$$\int_Ei\da (u_k-u)\wedge\db (u_k-u)\wedge i\da v\wedge\db v\leq\eta_k.$$
\end{lm}

\begin{proof}
Since $i\da v\wedge\db v\leq\dc v^2$, the Lemma follows directly from the previous Lemma.
\end{proof}

\begin{thr}\label{cap}
Let $\Omega\subset M$ be strictly pseudoconvex, $u\in\psh(\Omega)\cap L^\infty_{loc}$ and let $(u_k)\subset\psh(\Omega)$ be a sequence which decreases to $u$. Then $u_k\to u$ with respect to capacity.
\end{thr}

\begin{proof}
As usuall we can assume that $u_k=u$ outside some compact subset of $\Omega$. It is enough to prove that for $v\in\psh(\Omega)$, $0\leq v\leq1$ the integral
$$J_k=\int_\Omega(u_k-u)(\dc v)^2$$ converges to $0$ independently on $v$.
Using the definition of $(\dc v)^2$ and Stokes  theorem we get
$$J_k=-\int_\Omega {\rm Re} \dc v\wedge\da (u_k-u)\wedge \db v-2\int_\Omega {\rm Re}\da (u_k-u)\wedge \da v\wedge \tb\da v$$
$$-2\int_\Omega {\rm Re} (u_k-u) \tb\da v\wedge \ta\db v.$$
By the Cauchy-Schwarz inequality the first term  is bounded by $$\int_\Omega i\da (u_k-u)\wedge\db(u_k-u)\wedge\dc v$$ which converges to 0 by Lemma \ref{kolejnylematozbieznosci}. By Lemma \ref{L2mixzb} the second term converges to 0 independently of $v$ too. To estimate the third term, using (\ref{c_0}) we have
$$\tb\da v\wedge \ta\db v\leq c_0\dc v^2\wedge\omega$$
and we can do this exactly like in the proof of convergence with respect to $cap_\omega$ (Proposition \ref{capomega}).
\end{proof}

\section{Quasi-continuity and applications}\label{Quasi-continuity and applications}

As in the case of $\cn$ (and similarly as in section \ref{Pluripolarity}) as a direct consequence of Theorem \ref{cap} we obtain the quasi-continuity of bounded plurisubharmonic functions.

\begin{prp}\label{Quasi-continuity}Let $u\in L^\infty_{loc}\cap\psh(\Omega)$, where $\Omega$ is strictly pseudoconvex, and let $\varepsilon>0$. Then, there exists an open set $U\subset\Omega$ with $cap(U,\Omega)<\varepsilon$ and such that $u$ restricted to $\Omega\setminus U$ is continuous.
\end{prp}

As in the case of quasi-continuity with respect to $cap_\omega$ we get an appropriate convergence of currents.
\begin{crl}\label{frazypradygdyzbiegaja2}
Let $\mathcal{U}$ be a uniformly bounded family of plurisubharmonic functions in  $\Omega$. Suppose that $u,v,w\in\mathcal{U}$ and $(v_k),(w_k)\subset\mathcal{U}$ and $$\dc v_k\wedge\dc w_k\to \dc v.$$ Then $$u\dc v_k\wedge\dc w_k\to u\dc v\wedge\dc w.$$
\end{crl}

\begin{prp}\label{CP2}
Let $\Omega$ be almost Stein and $u$, $v\in\psh\cap L^\infty(\Omega)$. If 
$$\liminf_{z\to w}(u-v)\geq0,$$
for any $w\in\partial \Omega$, then
\begin{equation*}
    \int_{\{u<v\}}(\dc v)^2+2\tb\da v\wedge \ta\db v\leq\int_{\{u<v\}}(\dc u)^2+2\tb\da u\wedge \ta\db u.
\end{equation*}\end{prp}

This proposition can be proved   exactly in the same way as the Comparison principle in $\cn$. The  reason that we assume here that $\Omega$ is almost Stein is that we are able to regularize plurisubharmonic functions  only in such domains.

For $E\subset\Omega$ put
$$cap^\star(E)=\inf_{E\subset U} cap (U).$$
From definition we have
\begin{equation}\label{countablesubadditivity}
cap^\star (\bigcup E_i)\leq  \sum_i cap^\star( E_i),\end{equation}
for any countable family $\{E_i\}$ of subsets. 

\begin{lm}\label{plpbycap}
Let $\Omega$ be strictly pseudoconvex and $E\subset\Omega$. Assume that $cap^\star(E)=0$. Then there is a negative plurisubharmonic function $u$ on $\Omega$ such that $E\subset\{u=-\infty\}$.
\end{lm}
Of course a proof of this Lemma is very similar to the proof of Proposition \ref{plp}.
\begin{proof}
Let $U$ be an open subset of $\Omega$ and let $K_i$ be an increasing sequence of compact subsets with $\bigcup K_i=\bigcup int K_i=U$. The sequence $u_{K_i}^\star$ decreases to $u_{U}^\star$ and by the convergence theorem for decreasing sequences we have $$\int_\Omega(\dc u_{U}^\star)^2\leq cap (U).$$

Let $E_i$ be an increasing sequence of relatively compact subsets with $\bigcup E_i=E$. Fix $i$. We can choose a decreasing sequence of relatively compact open subsets $U_k$ such that $cap(U_k)\to 0$. Thus 
$$\int_\Omega(\dc u_{U_k}^\star)^2\to 0.$$ By the convergence theorem for increasing sequences and Lemma \ref{CP} the sequence $u_{U_k}^\star$ converge almost everywhere to 0 and  we can choose $k$ such that
$$\|u_i\|_{L^1(\omega)}\leq\frac{1}{2}$$
for $u_i=u_{U_k}^\star$. Since $u_i\leq-1$ on $E_i$ we can conclude that a function 
$$u=\sum_i u_i$$ has the required properties.\end{proof}

For an open set $U\subset\Omega$ obviously we have $cap^\star (U)=cap (U)$. For compact subsets we will prove the following
\begin{lm}\label{cap0dlazwartych}
 Let $\Omega$ be strictly pseudoconvex and let $E\subset\Omega$ be compact. Then if $cap (E)=
0$ then $ cap^\star (E)=0$.
\end{lm}

\begin{proof}
By Proposition \ref{plp} there is a negative plurisubharmonic function $u$ such that $E\subset\{u=-\infty\}$. Let $U\Subset\Omega$ be a neighbourhood of $E$,   $U_k=\{u<-k\}\cap U$ and $u_k=3u_{U_k}^\star+1$. Since $1\geq u_k\geq \frac{3}{k}u+1$, the sequence $u_k$ increases to 1 almost everywhere. In particular by Corollary \ref{w12dlarosnacych} and Theorem \ref{MAdlarosnacych} we infer
$$\lim_{k\to\infty}\int_V(\dc u_k)^2=\lim_{k\to\infty}\int_V\tb\da u_k\wedge \ta\db u_k=0,$$
where $V=\{u_1<0\}.$ For any $h_k\in\psh(\Omega)$ such that $-1\leq h_k\leq 0$ by Proposition \ref{CP2} we can estimate
$$\int_{U_k}(\dc h_k)^2\leq\int_{\{u_k<h_k\}}(\dc h_k)^2+\tb\da h_k\wedge \ta\db h_k$$
$$\leq\int_{\{u_k<h_k\}}(\dc u_k)^2+\tb\da u_k\wedge \ta\db u_k\leq\int_V(\dc u_k)^2+\tb\da u_k\wedge \ta\db u_k\to 0.$$
Thus
$$cap (U_k)\to 0$$ and the Lemma follows.
\end{proof}

The use of Proposition \ref{CP2} was not necessary in the above proof. We could just use Stokes Theorem like before.

A set $E\subset\Omega$ is called negligible if $E\subset\{u<u^\star\}$ for $u$ a supremum of family of plurisubharmonic functions. As in $\cn$ it is clear that pluripolar sets are (at least locally) negligible. Now we  prove the converse.

\begin{thr}\label{negligible}
Let $\Omega$ be strictly pseudoconvex and let $E\subset\Omega$ be negligible. Then there is a negative plurisubharmonic function $u$ on $\Omega$ such that $E\subset\{u=-\infty\}$.
\end{thr}

\begin{proof}
Let $u=\sup u_\alpha$, where $(u_\alpha)\subset\psh(\Omega)$ and $E\subset\{u< u^\star\}$. By Choquet's lemma we can assume that $u=\lim u_k$ where $u_k$ is an increasing sequence of plurisubharmonic functions. Let $(K_i)_{i\in\N}$ be a sequence of compact subsets such that $\Omega=\bigcup K_i$. By  (\ref{countablesubadditivity}) and Lemma \ref{plpbycap} it is enough to show that $cap^\star(A)=0$ for 
$$A=A_{its}=K_i\cap\{u\leq t<s\leq u^\star\},$$
for all $i\in\N,t,s\in\Q$ and $t<s$.

Set $\varepsilon>0$. We can assume that $u_k$ are bounded (if not we put new $u_k$ as $\max \{u_k,t-1\}$ and the set $A$ does not change). By the quasi-continuity we can choose an open subset $U\subset\Omega$ such that $cap(U)<\varepsilon$ and all functions $u_k$ and $u^\star$ are continuous on $\Omega\setminus U$. Then $B=A\setminus U$ is compact. By Lemma \ref{cap0dlazwartych} and because $\varepsilon$ is arbitrary it is enough to show that $cap(B)=0$.

Let $h\in\psh\cap L^\infty(\Omega)$. We will show that
\begin{equation}\label{capnegligibe}
    \int_B(\dc h)^2=0.
\end{equation} 
Since $B$ is compact we can assume that $u_k=u^\star$ outside some compact subset of $\Omega$. We  have
$$\int_B(\dc h)^2\leq(t-s)^{-1}\int_\Omega (u-u_k)(\dc h)^2=(t-s)^{-1}I_k.$$ 
To estimate $I$ we will use integration by parts (in particular we use here Lemma \ref{doStokesa})
$$I_k=\dc(u_k-u)\wedge i\da h\wedge\db h$$ $$-\da(u-u_k)\wedge\da h\wedge\tb\da h-\db(u-u_k)\wedge\ta\db h\wedge\db h
-2(u-u_k)\tb\da h\wedge\ta\db h.$$
To estimate the first term  note that
$$\dc(u_k-u)\wedge i\da h\wedge\db h=\frac{1}{2}\dc(u_k-u)\wedge \dc h^2-h\dc(u_k-u)\wedge i\dc h,$$ 
and by Lemma \ref{niekonieczniepsh} and Corollary \ref{frazypradygdyzbiegaja2} it converges to 0. The last term converges by the Lebesgue monotone convergence theorem. To estimate the second term we can choose a smooth function $h'$ close in $W^{1,2}$ to $h$. By by the Cauchy-Schwarz inequality and Lemma \ref{L2mix} the expression
$$\da(u-u_k)\wedge\da h\wedge\tb\da (h-h')$$
is close to 0 (independently on $k$). Thus the convergence to 0 of the second term is the consequence of the convergence of
$$\da(u-u_k)\wedge\da h\wedge\tb\da h'.$$
In a similar way we get that 
$$-\db(u-u_k)\wedge\ta\db h\wedge\db h\to 0$$
and the Theorem follows.
\end{proof}

The above Theorem and Lemma \ref{plpbycap} give the following
\begin{crl}\label{plpbycap2}
Let $\Omega$ be strictly pseudoconvex and $E\subset\Omega$. Then  $cap^\star(E)=0$ iff $E$ is pluripolar. In particular a countable union of pluripolar sets is pluripolar.
\end{crl}

Another important corollary is the quasi-continuity of not necessarily bounded plurisubharmonic functions.

\begin{thr}\label{Quasi-continuity2}
Let $u\in \psh(\Omega)$, where $\Omega$ is strictly pseudoconvex, and let $\varepsilon>0$. Then, there exists an open set $U\subset\Omega$ with $cap(U,\Omega)<\varepsilon$ and such that $u$ restricted to $\Omega\setminus U$ is continuous.
\end{thr}
\begin{proof}
By proposition \ref{Quasi-continuity} a function $e^u$ is quasicontinuous and thus there is an open set $U_1$ such that $cap(U_1,\Omega)<\varepsilon/2$ and  $e^u$ restricted to $\Omega\setminus U_1$ is continuous. By Corollary \ref{plpbycap2} there is an another open set $U_2$ for which $cap(U_2,\Omega)<\varepsilon/2$ and $\{u=-\infty\}\subset U_2$. The union $U=U_1\cup U_2$ is the set we are looking for.
\end{proof}

Using Corollary \ref{plpbycap2}, in the same way as Proposition \ref{globalplp} we can prove the following Josefson type Theorem.

\begin{thr}\label{globalplp2}
Let $M$ be an almost Stein manifold and let $E\subset M$ be pluripolar set. Then there is a plurisubharmonic function $u$ on $\Omega$ such that $u|_E=-\infty$.
\end{thr}

\section{Open Problems}

The most important problem here is to built the pluripotential theory in higher dimension. However the theory on surfaces is not complete too. Let us give here three interesting open questions.

{\bf Question 1.} Does the domination principle hold for the Monge-Amp\`{e}re operator on almost complex surfaces? More precisely:  assume that for $u,v\in\psh\cap L^\infty(\Omega)$ we have  $(\dc u)^2\geq(\dc v)^2$ on $\Omega$. Does $u\leq v$ outside compact subset of $\Omega$ imply that $u\leq v$ in $\Omega$? The answer is unknown even in the case of continuous functions (see \cite{p3} for some partial result in this direction).

{\bf Question 2.} Let $\Omega$ be strictly pseudoconvex and $E\Subset\Omega$. Is there any upper estimate of $cap^\star(E)$ in terms of $\int_\Omega(\dc u^\star_E)^2$?

{\bf Question 3.} Let $u\in\psh$ and $v\in\psh\cap L^\infty$. Is the current  $u(\dc v)^2$ well defined?

$\newline$
{\bf Acknowledgments}
The authors would like to thank S\l awomir Dinew and Grzegorz Kapustka for the useful conversations.

\end{document}